\documentclass[a4paper,reqno]{amsart}
\usepackage[english]{babel}
\usepackage{mathtools}
\usepackage[shortlabels]{enumitem} 

\usepackage[colorlinks,citecolor=blue]{hyperref}

\usepackage[capitalize]{cleveref} 
\usepackage[nobysame, lite]{amsrefs}

\usepackage{tikz-cd} 
\newtheorem{theorem}{Theorem}[section]
\newtheorem{proposition}[theorem]{Proposition}
\newtheorem{definition}[theorem]{Definition}
\newtheorem{lemma}[theorem]{Lemma}
\newtheorem{corollary}[theorem]{Corollary}
\newtheorem{remark}[theorem]{Remark}

\numberwithin{equation}{section}
\numberwithin{theorem}{section}

\newcommand{\bel}[2]{\begin{equation} \label{#1} \begin{split} #2
\end{split} \end{equation}} 
\newcommand{\mc}[1]{{\mathcal #1}}

\newcommand{\mb}[1]{{\mathbf #1}}
\newcommand{\bb}[1]{{\mathbb #1}}
\newcommand{\st}{\,:\:}
\newcommand{\eps}{\varepsilon}

\newcommand{\HH}{\mathfrak{H}}
\newcommand{\vol}{\operatorname{Vol}}

\usepackage{bm}
\newcommand{\sjl}{{\bm{(}}}
\newcommand{\sjd}{{\bm{)}}}

\author[A.Galkin]{Artem Galkin}
\address{Faculty of Mathematics, National Research University Higher School of Economics, 6 Usacheva St., 119048 Moscow, Russia.}
\email{agalkin@hse.ru}

\author[M.Mariani]{Mauro Mariani}
\address{Faculty of Mathematics, National Research University Higher School of Economics, 6 Usacheva St., 119048 Moscow, Russia.}
\email{mmariani@hse.ru}

\title[Random homology of diffusion processes]{Asymptotic properties of random homology induced by diffusion processes}

\begin{document}
\keywords{Random homology, Asymptotic of diffusion processes, Albanese map, Gallavotti-Cohen symmetry, Weighted harmonic maps}
\subjclass[2010]{58J65,	60H10,60F10,53C43}

\begin{abstract}
	We investigate the asymptotic behavior, in the long time limit, of the random homology associated to realizations of stochastic diffusion processes on a compact Riemannian manifold. In particular a rigidity result is established: if the rate is quadratic, then the manifold is a locally trivial fiber bundle over a flat torus, with fibers being minimal in a weighted sense (that is, regarding the manifold as a metric measured space, with the invariant probability being the weight measure). Surprisingly, this entails that at least for some classes of manifolds, the homology of non-reversible processes relaxes to equilibrium \emph{slower} than its reversible counterpart (as opposed to the respective empirical measure, which relaxes faster).
\end{abstract}

\maketitle

\section{Introduction}
\label{s:intro}
The asymptotic behavior of invariants of random curves on manifolds is a classical subject, a famous example being the Cauchy asymptotic of the winding number of a Brownian motion in $\bb R^2$ due to Spitzer, see  \cite{spitzer1958some}.
Manabe, Lyons \& McKean, and Pitman \& Yor further investigated winding numbers in a well-known series of papers throughout the '80s, \cite{manabe1982stochastic,lyons1984winding,pitman1984asymptotic,pitman1986asymptotic,pitman1989further}, with some large deviations bounds provided in \cite{avellaneda1985large}. It is not possible to mention all the literature starting from the '90s, even when considering just Brownian winding numbers. A geometrical approach to the problem has been developed by the Japanese school of stochastic calculus, an explanatory reference being Watanabe's \cite{watanabe2000asymptotic}, where the asymptotic behavior of windings around points and windings around topological "holes" is investigated on Riemannian surfaces.

\subsubsection*{Results and motivations}

This paper is mostly motivated by this approach, in particular it studies the qualitative behavior of the fluctuations of the random homology associated to diffusion processes on compact manifolds. We focus on the long time limit $T\to \infty$, of the random homology $h_T$ associated to paths of general diffusion processes. While precise definitions are given in Section~\ref{s:preliminaries}, this should be thought ash a finite-dimensional random variable that describes the number of windings per unit time of a stochastic continuous curve $(X_t)_{t \le T}$. It is easy to show that in the sense of large deviations, as $T\to \infty$
\bel{e:ld}{
   \bb P(h_T \simeq h) \simeq \exp\left(-T G(h) \right)
}
where $G(h) \in [0,\infty]$ is a positive \emph{rate function} defined on the space of DeRham homologies $h\in H_1(M;\bb R)$, see Section~\ref{ss:rate}. For a fixed compact Riemannian manifold $M$, the law of the diffusion process $X_t$ is identified by the Riemannian metric $g$ and a tangent vector field $b$ on $M$. Thus $G(\cdot)$ is itself determined by $g$ and $b$ (while it is independent of the initial condition $X_0$), and our goal is to establish for which triples $(M,g,b)$, the rate $G$ features some qualitative properties. To fix the ideas, consider the SDE on a flat $d$-dimensional torus
\bel{e:flat_sde}{
   \dot X= \bar h+\dot B
}
where $\bar h\in \bb R^d$ is a constant and $B$ is a standard (Gaussian) Brownian motion. Then the vector $h_T\coloneq (\tilde{X}_T-\tilde{X}_0)/T$ describes exactly the time-normalized number of times the process winded around holes of the torus, where $\tilde{X}$ is the lift of $X$ to $\bb R^d$ (see Remark~\ref{r:ldh} about the fact that $(X_t)$ is not a closed curve). In this case $h_T$ is a Gaussian random variable with expected value $\bar h$, and it is easily seen that as $T\to \infty$ it satisfies a large deviations principle with a \emph{quadratic} rate $G(h)=\tfrac{1}{2} \|h-\bar h\|^2$ for a suitable norm $\|\cdot\|$. In general, on any manifold, as $T\to \infty$, $h_T$ converges to some deterministic $\bar h\in H_1(M;\bb R)$, and $G(h)$ describes the asymptotic of $h_T$ and in particular it quantifies how unlikely it is to fluctuate away from $\bar h$ (thus particular $G(h)=0$ only for $h=\bar h$).

Our main achievement is a rigidity result stating that if $G$ is quadratic, then the manifold $M$ is not too different from a flat torus with a random dynamic of type \eqref{e:flat_sde}. While the precise statement of Theorem~\ref{t:asympnr} requires some preliminaries, it basically states that a triple $(M,g,b)$ featuring a quadratic rate functional $G(\cdot)$ for the associated random homology $h_T$, is necessarily a locally trivial fiber bundle over a flat torus $\bb T^{b_1}$ (where $b_1$ is the first Betti number), moreover fibers are minimal in a suitable sense (which corresponds to minimality on the metric measured space $(M,g,m)$ where $m$ is in the invariant measure of the process, see Section~\ref{ss:fib}), and some weak reversibility property should finally hold. While there are trivial examples of such manifolds and diffusion processes (e.g.\ \eqref{e:flat_sde} or a product of \eqref{e:flat_sde} and a simply connected manifolds with independent random dynamics), also nontrivial examples exists. In particular, we deduce that a Brownian motion on formal manifolds, see \cite{Kotschick2001}, features such a quadratic rate; or more in general one may build weighted formal manifolds with such a   property.

Our result has a nontrivial consequence. As discussed in detail in Remark~\ref{r:nonrevslower}, one usually expects that non-reversible Markov processes present a faster convergence to equilibrium (and thus smaller fluctuations and larger deviations rates), when compared to their reversible counterpart (that is, the process associated to the symmetric part of their generator). This folklore idea can be made rigorous in various contexts, for instance:  through spectral analysis on finite graphs; or, exactly in the aforementioned sense of large deviations, for the occupation measure of diffusion processes, see \cite{rey2015irreversible,rey2016improving}. This is a rather significant phenomenon in applications, since it may be used to speed-up the convergence rate of MCMC algorithms. The actual rigorous result is that the large deviations rate of observables that are functions of the occupation measure of a Markov process, is larger than the rate of the same observables associated to their reversible part, thus yielding smaller fluctuations and faster convergence in the non-reversible case.

Random homologies provide a counterexample to this phenomenon. The occupation measure of a diffusion process can be recovered from the action of the stochastic current, see \eqref{e:pit} on exact forms; however the random homology associated to a process is obtained by the action of the same random current  on \emph{closed} forms. Surprisingly, as soon as one extends in such a minimal way (from exact to closed) the set of observables considered, the previously described inequality on deviations rates is broken, and even more, when considering the quotient closed/excact, the \emph{opposite} inequality actually holds. In a suitable sense reversible processes (or more in general homologically reversible processes, see Definition~\ref{d:quasireversible}) are the ones presenting the fastest convergence rates. This is a straightforward consequence of Theorem~\ref{t:asympnr}, and while not in contradiction with \cite{rey2015irreversible}, it is certainly unexpected, see Remark~\ref{r:nonrevslower}.

A second set of results, see Proposition~\ref{p:bound}, computes the expansion of $G(h)$ around its minimizer $\bar h$, that is we compute $\lim_{\eps \to 0} \eps^{-2} G(\bar h+\eps h)$, a global explicit bound $G(h)\le Q(h)$, and we establish a Gallavotti-Cohen symmetry for $G$ in some cases. Informally speaking, the large deviations rate functions of the pair measure-current are the out-of-equilibrium counterpart of thermodynamic potentials in the context of equilibrium Statistical Mechanics and Dynamical Systems, see e.g.\ \cite{bertini2015macroscopic} and references there in. In this context, the significance of Gallavotti-Cohen symmetries for the rate function of the pair measure-current for non-reversible processes has been a subject of interest, likely after \cite{lebowitz1999gallavotti}. More recently, such symmetries have been noticed (to hold, or often not to hold) for some observables informally related to discrete homologies,  \cites{faggionato2011gallavotti,faggionato2017random}.

In Section~\ref{ss:open} we further discuss some deeper motivations that require some additional preliminaries to be properly discussed.

\subsubsection*{Plan of the paper}
The paper is organized as follows. In Section~\ref{s:preliminaries} we introduce the main notation and recall some results concerning random currents. In Section~\ref{s:main} we state our main results, namely Proposition~\ref{p:bound} and Theorem~\ref{t:asympnr}, and discuss some open problems. In Section~\ref{s:tools} we introduce some mathematical tools that may have an independent interest, in particular weighted Albanese maps, the relation between weighted minimality and weighted harmonicity and the lift of Markov generators. In Section~\ref{s:proofquad}-\ref{s:proofag} we prove the main results together with some additional statements which may have a broader interest.

\subsubsection*{Acknowledgement} We are thankful to Domenico Fiorenza for pointing out Corollary~\ref{c:torelli}.

\section{Preliminaries}
\label{s:preliminaries}

\subsection{Notation}
\label{ss:notation}
Let $(M, g)$ be a smooth closed compact, connected Riemannian manifold. $\mc D^k \equiv \mc D^k(M)$ denotes the space of smooth $k$-forms on $M$, $d$ and $d^\ast$ the differential and codifferential. In particular we mostly consider $\mc D^0$ (smooth functions) and $\mc D^1$ (smooth $1$-forms). $\mc P(M)$ denotes the space of Borel probability measures on $M$ endowed with the standard narrow topology, and $\mc J(M)$ the space of $1$-currents regarded as the dual of $\mc D^1$.  $\langle \cdot, \cdot \rangle$ stands for the pairing between vector fields and $1$-forms on $M$. In particular we understand, as standard in the probabilistic notation, $\langle b,df\rangle$ as the action $bf$ of the vector field $b$ on $f\in \mc D^0$. $\mu(f)$ denotes the integral of $f$ w.r.t.\ $\mu$, and $j(\omega)$ the action of a current $j$ on $\omega$.

Since we have fixed a Riemannian tensor $g$, we denote without further notice $|\cdot|$ the associated norms both on tangent and cotangent spaces. For instance, with this notation, $|\omega|^2(x)= \langle g^{-1} \omega, \omega\rangle(x)$. With an abuse of notation, $\|\cdot\|_\mu$ denotes the induced $L^2(\mu)$-norms both on $\mu$-square integrable $1$-forms and currents, e.g.\  $\|\omega\|_{\mu}^2\coloneq \mu(| \omega|^2)$.

$H^1(M;\bb R)$ and $H_1(M;\bb R)$ denote the space of De Rham cohomology and its dual, the real homology group of the manifold $M$. $[c]$ stands for the space of closed $\omega \in \mc{D}^1$ in cohomology class $c$, and $\langle h,c\rangle$ denotes the homology-cohomology duality\footnote{As various scalar products are used in this paper, we use a notation which is more common in the probabilistic literature but somehow unusual in differential geometry. Angled brackets $\langle \cdot , \cdot \rangle$ are used for dualities that are independent of the metric $g$, while $(\cdot,\cdot)$, $(\cdot,\cdot)_r$, $\sjl \cdot,\cdot \sjd$ denote scalar products that do depend $g$.}.

Recall that a Riemannian metric is defined on $M$, let $\Delta$ be the associated Laplace-Beltrami operator and $b$ a smooth vector field on $M$. Let $(X_t)_{t \geq 0}$ be the $M$-valued Feller process with  generator $L$ defined on smooth functions as
\bel{e:generator}{
	Lf= \tfrac 12  \Delta f + \langle b,df \rangle
}
There exists a unique probability measure $m$ such that $m(Lf) = 0$ for any $f \in \mc D^0$, referred to as the \emph{invariant measure}. With a little abuse of notation, hereafter we still denote by $L$ the closure of $L$ both in $C(M)$ and $L^2(m)$. Finally, for $\mu\in \mc P(M)$, we denote $j_\mu\in \mc J(M)$ the \emph{typical current} defined by
\bel{e:jmu_def}{
j_\mu(\omega)\coloneq\mu(\tfrac 12 d^\ast \omega + \langle b,\omega\rangle)
}

\subsection{Rate function for random homologies}
\label{ss:rate}
In this section we define the rate function $G$ for random homologies. As it is easily derived from known results, we define it rigorously but concisely, and refer to Section~\ref{s:tools} for further details. Recall that the operator $L$ defined in \eqref{e:generator} generates a Markov process. The empirical measure $\pi_T \in \mc P(M)$ and the empirical current $J_T \in \mc J(M)$ are then defined pathwise as ($\circ$ denotes the Stratonovich integral)
\bel{e:pit}{
	\pi_T(f)\coloneq \tfrac{1}T \int_0^T f(X_s)\,ds, \qquad f \in \mc D^0
	\\
	J_T(\omega) = \tfrac 1T \int_0^T \omega(X_s) \circ dX_s, \qquad \omega \in \mc D^1
}
To be precise, \eqref{e:pit} does not immediately identify a random element $J_T\in \mc J(M)$, since the definition holds only a.e.\ for each $\omega \in \mc D^1$, but $\mc D^1$ is uncountable. This technicality can be overcome in various equivalent ways, either defining $J_T$ for \emph{every} $X \in C^\alpha(M)$ as a geometric rough path, see e.g.\ the seminal \cite{lyons1998differential}, or using a classical result by Mitoma \cite[page~358]{carmona1986introduction}. In any case \eqref{e:pit} provides a well-posed definition of random current. Sharper results of well-posedness $J_T$ in stronger space are possible, see for instance \cites{flandoli2005stochastic,flandoli2009regularity}.

Large deviations results, in the limit $T\to \infty$, can be established for the pair $(\pi_T,J_T)$  with standard tools. In the last decades, sharper techniques strengthened  the topology in which large deviations hold, in the same way as rough paths techniques strengthened the topology for the well-posedness of $J_T$, see \cites{kuwada2003sample,kusuoka2010large} with the full results for generic diffusions in \cite{kuwada2006large,galkin2024large}. These results motivate the following definitions.

\begin{definition}[Large Deviations Rates]
	\label{d:h}
	$\HH$ denotes the space of pairs $(\mu,j) \in \mc P(M)\times \mc J(M)$ such that
	\begin{itemize}
		\item $j$ is a closed current, namely $j(df)=0$ for all $f\in \mc D^0$.
		\item $\mu$ has finite Fisher information w.r.t.\ the invariant measure $m$, namely $\mu$ is absolutely continuous $\mu = \varrho m$ and $\sqrt{\varrho} \in W^{1,2}(M,m)$.
	\end{itemize}

	The \emph{large deviations rate of the pair empirical measure-current} is defined as
	\bel{e:rate}{
		&  I \colon \mc P(M) \times \mc J(M) \to [0,\infty]
		\\
		&	I(\mu,j)\coloneq
		\begin{cases}
			\tfrac 12 \|j-j_\mu\|^2_\mu & \text{if $(\mu,j) \in \HH$}
			\\
			+\infty                     & \text{otherwise}
		\end{cases}
	}
	The \emph{large deviations rate $G$ of the random homology} associated to the generator $L$ in \eqref{e:generator} is defined as
	\bel{e:rate2}{
		& G\colon H_1(M;\bb R) \to [0,\infty]
		\\
		& G(h)\coloneq \inf\big\{   I(\mu,j)\st  j(\omega)= (c,h),\,\forall \, \omega \in [c]
		\big\}
	}
\end{definition}
As discussed before, the law of $(\pi_T,J_T)$ satisfies a large deviations principle with speed $T$ and rate $I$, see the aforementioned \cites{galkin2024large,kuwada2003sample,kusuoka2010large} for details. However, since paths of diffusion processes are not closed, the restriction of $J_T$ to closed $1$-forms does not identify a random homology. Fix however any linear isomorphism $H^1(M, \bb R) \ni c \mapsto \xi_c \in \mc D^1$, associating to each cohomology class $c \in H^1(M, \bb R)$ a closed $1$-form $\xi_c \in [c]$. Then a random homology $h_T\in H_1(M;\bb R)$ is naturally associated to the Markov process $X$ by the relation
\bel{e:randomh}{
	\langle h_T,c \rangle\coloneq  J_T(\xi_c), \qquad c\in H^1(M;\bb R)
}
The following remark is an immediate consequence of the contraction principle \cite[Chapter~4.2.1]{dembo2010} and standard properties of geometric rough paths/Stratonovich integrals, and it is quickly proved below.
\begin{remark}
	\label{r:ldh}
	Regardless of the isomorphism $c \mapsto \xi_c$, the law of $h_T$ satisfies a \emph{good} large deviations principle with speed $T$ and rate function $G$.
\end{remark}
The main results of our paper concern the qualitative behavior of the rate function $G(\cdot)$ in terms of the topological properties of the manifold $M$ and the reversibility and curvature properties of the generator $L$ in \eqref{e:generator}.

\section{Main result}
\label{s:main}
In this section we state our main results, Proposition~\ref{p:bound} and Theorem~\ref{t:asympnr}.

\subsection{Quadratic bounds and symmetries}
\label{ss:gc}
Recall that $g$ denotes the metric tensor and $m$ the invariant measure. It is a standard fact that the invariant measure has a smooth, strictly positve density w.r.t.\ to the volume measure on $M$ and we write
\bel{e:vr}{
& m=e^{-V} \, \vol
\\
& r \coloneq b+ \tfrac{1}{2} \nabla V
}
$r$ may be interpreted as the non-reversible part of the drift $b$, see Definition~\ref{d:quasireversible}.

The metric-measure triple $(M,g,m)$ induces a scalar product on real homologies and cohomologies as follows. For each $c\in H^1(M;\bb R)$, there exists a unique closed $1$-form $\eta_c \in [c]$ that is orthogonal to exact forms in $L^2(m)$ (see Section~\ref{ss:dual}). Then $(c,c')\coloneq m( \langle g^{-1}\eta_c,\eta_{c'}\rangle)$ defines a scalar product on $H^1(M;\bb R)$ and $(h,h')$ denotes the dual scalar product on $H_1(M;\bb R)$. Similarly, for each $c\in H^1(M;\bb R)$, there exists a unique $\mb L$-harmonic form $\omega_c \in [c]$, see Definition~\ref{d:mharnom} and Remark~\ref{r:isomorphism}, and we can introduce a second scalar product $(c,c')_r\coloneq m( \langle g^{-1} \omega_c,\omega_{c'}\rangle)$ on $H^1(M;\bb R)$, and $(h,h')_r$ denotes the dual scalar product on $H_1(M;\bb R)$. The index $r$ stresses the dependence on the non-reversible field $r$, see \eqref{e:vr}. In general, see Remark~\ref{r:isomorphism}, $(h,h)_r \le (h,h)$ and the two scalar products coincide in the reversible case $r=0$.

We start defining some relevant extensions of reversibility.
\begin{definition}
	\label{d:quasireversible}
	The generator $L$ (and the vector field $b$) is called
	\begin{enumerate}[label=(\alph*)]
		\item \label{(a)} \emph{reversible}: if it is self-adjoint in $L^2(m)$. Equivalently, if the $1$-form $gb$ is exact, or yet equivalently if the non-reversible field $r$ vanishes, see \eqref{e:vr}.
		\item \label{(b)} \emph{quasi-reversible}: if $gb$ is closed. Equivalently, if the $gr$ is $m$-harmonic.
		\item \label{(c)} \emph{homologically reversible}: if the scalar products introduced above coincide, $(c,c)=(c,c)_r$. Or equivalently, if $r$ satisfies $\langle r,\eta_c\rangle=\operatorname{const}$, for all $c\in H^1(M;\bb R)$.
		      Here $\operatorname{const}$ means that the scalar product is independent
		      of $x \in M$.
		\item \label{(d)} \emph{typically reversible}: if $m(\langle r,\eta_c \rangle)=0$ for every $c\in H^1(M;\bb R)$; namely if the restriction of $J_T$ to closed $1$-forms vanishes as $T\to \infty$.
	\end{enumerate}
\end{definition}
Of course reversibility implies quasi-reversibility, homological and typical reversibility.
On the other hand \ref{(b)} + \ref{(c)} + \ref{(d)} is actually equivalent to reversibility (as easy to show but not needed for the following).
One may check that \ref{(b)} implies \ref{(c)} on the same class of manifolds characterized by the conditions in Theorem~\ref{t:asympnr}-(\ref{a:mminnr}).

\begin{proposition}
	\label{p:bound}
	There exists $\bar h \in H_1(M,R)$, that only depends on $g$ and $b$, such that
	\bel{e:gaussineq}{
		G(h) \le Q(h)\coloneq\tfrac{1}{2} (h -\bar h,h -\bar h)
	}
	and
	\bel{e:eps}{
		\lim_{\eps \downarrow 0} \eps^{-2}\,G(\bar h+\eps\,h)= \tfrac{1}{2} (h,h)_r
	}
	Moreover in the quasi-reversible case (see Definition~\ref{d:quasireversible}-\ref{(b)}), for $\bar c \in H^1(M;\bb R)$ the cohomology class of $gb$, $G$ satisfies
	\bel{e:gc}{
		G(h)-G(-h)=Q(h)-Q(-h)=-2\langle h,\bar c\rangle
	}
\end{proposition}
The last proposition shows in particular that
\begin{itemize}
	\item For quasi-reversible processes,  $G$ enjoys a Gallavotti-Cohen symmetry, actually the same Gallavotti-Cohen symmetry that $Q$ trivially satisfies.

	\item For homologically reversible processes, $Q$ provides both a global upper bound and the small-homology limit of $G$, and also in this case \eqref{e:gaussineq} writes $G(h)\le \tfrac 12 \langle G''(\bar h)(h-\bar h) ,  h-\bar h \rangle$.

	\item $G(h)=0$ if and only if $h=\bar h = \lim_{T\to \infty} h_T/T$. $\bar h$ is actually the rotation number of the current $mr$, see Remark~\ref{r:hbar}.

	\item For typically reversible processes, $G$ is minimized at $h=0$ and it is quadratic around $0$.
\end{itemize}
More generally the quadratic bound in \eqref{e:gaussineq} can be interpreted as a sub-gaussian bound for the large deviations of the random homology $h_T$ introduced in \eqref{e:randomh}. In the next Section~\ref{ss:ag} we fully characterize the cases where equality holds, the motivation for such a question is briefly explained in Section~\ref{ss:open} below.

\subsection{Asymptotically Gaussian homologies}
\label{ss:ag}
We say that the diffusion process $X$ associated to the generator $L$ has \emph{asymptotically Gaussian} homology if equality holds in \eqref{e:gaussineq}, that is
\bel{e:gauss2}{
	G(h)= Q(h) \coloneq \tfrac 12 (h -\bar h,h - \bar h)
}
This wording is due to the fact that, whenever the restriction of $J_T$ to harmonic $1$-forms is Gaussian (for instance on a flat torus), the equality in \eqref{e:gaussineq} holds indeed. In this section we characterize asymptotically Gaussian homologies via a topological rigidity condition and an equivalent stochastic interpretation. Informally speaking, an asymptotically Gaussian homology can only rise from an underlying diffusion on flat tori.

The metric-measure space $(M,g,m)$ is naturally associated to the generator $L$: indeed one can recover $g$ from $L$, see e.g.\  \cites{ambrosio2005gradient,bakry2014}, while $m$ is just the invariant measure of $L$. We then borrow the
notion of $m$-weighted minimality from such a metric-measure context, see \cites{cheng2015stability,cheng2020minimal}, and we conveniently rephrase it here in our framework.
\begin{definition}
	\label{d:minimal}
	For $e^{-V}$ the density of $m$ w.r.t.\ the volume measure on $M$ as above, and for $N$ a submanifold with induced volume $\sigma$, $N$ is called $m$-minimal if any of the following equivalent conditions is satisfied
	\begin{enumerate}
		\item $N$ is a stationary point of the $m$-volume functional $N\mapsto  \int_N e^{-V} d\sigma$.
		\item The $m$-mean curvature $\mb H_m \coloneq \mb H+\nabla V^\perp$ vanishes, where $\mb H$ is the usual mean curvature and $\nabla V^\perp$ is the normal (to $N$) projection of $\nabla V$.
	\end{enumerate}
	If the dimension $d\neq 1,2$, the previous conditions are also equivalent to
	\begin{enumerate}
	\item $N$ is minimal w.r.t.\ the conformally equivalent metric $g'=e^{2V/d}g$.
	\end{enumerate}
\end{definition}

We finally state our main result. In the following theorem and hereafter, $b_1 \equiv b_1(M)$ is the first Betti number of $M$.
\begin{theorem}
	\label{t:asympnr}
	The following are equivalent.
	\begin{enumerate}
		\item\label{a:agnr} The generator $L$ has asymptotically Gaussian homology, see Section~\ref{e:gauss2}.
		\item \label{a:mminnr} $M$ is a locally trivial fiber bundle over a flat torus $\bb T^{b_1}$, with $m$-minimal fibers, and $b$ is homologically reversible.
		\item\label{a:bmnr} There exists a smooth map $\phi\colon M\to \bb T^{b_1}$ and \footnote{It follows from the proof that $\tilde h$ is characterized by $\bar h$ introduced in Section~\ref{ss:gc}, up to a linear transformation. In particular $\tilde h=0$ for typically reversible processes, see Definition~\ref{d:quasireversible}.} $\tilde h\in H_1(\bb T^{b_1};\bb R)$ such that $Y=\phi(X)$ is a solution to the SDE
		      \bel{e:sdetorus}{
			      dY_t=\tilde h\,dt+dW_t
		      }
		      where $W$ is a Brownian motion on $\bb T^{b_1}$ w.r.t.\ a flat metric on $\bb T^{b_1}$.
	\end{enumerate}
	In particular $b_1\le \mathrm{dim}(M)$ and if $b_1=\mathrm{dim}(M)$ then $M$ is a flat torus and $X$ solves \eqref{e:sdetorus}  with $\tilde h=\bar h$.
\end{theorem}
In some sense, the previous theorem contains two statements. The first one is purely geometrical: asymptotic gaussianity of the homology is equivalent to the metric measure space $(M,g,m)$  being a locally trivial fiber bundle over a flat torus in a way depending specifically on the weight $m$. The second states that only homologically reversible processes can have asymptotically gaussian homologies. For instance, in the reversible case, the theorem boils down to a differential geometric characterization: a reversible process has asymptotically gaussian homology if and only if $M$ is a locally trivial fiber bundle over a flat torus $\bb T^{b_1}$, with $m$-minimal fibers.

We remark that there exist non-trivial examples of manifolds $M$ endowed with a diffusion generator $L$ satisfying any of the equivalent conditions of Theorem~\ref{t:asympnr}. A trivial example is a product of a flat torus with diffusion as in Thoerem~\ref{t:asympnr}-\ref{a:bmnr} with a simply connected manifold and an independent random dynamics. However, even in the case of Brownian motion $b=0$, one can show non-trivial examples such as formal manifolds, see \cite{Kotschick2001}. Indeed harmonic forms on formal manifolds have constant length, a property which implies  Theorem~\ref{t:asympnr}-\ref{a:mminnr} (provided $b=0$), as proved in \cite[Proposition~5]{nagy2004length}.

\begin{remark}
	\label{r:nonrevslower}
	It is folklore knowledge that non-reversible processes converge to equilibrium faster than their symmetric (reversible) part. In our framework, this means that as $t\to \infty$, one expects the process $X_t$ with generator \eqref{e:generator} to converge in law to its stationary limit, faster than a process, say $Y_t$, with generator $L'f=\tfrac{1}{2}\Delta f - \langle \nabla V,df\rangle$, where $V$ is defined as in \eqref{e:vr}. In other words, $X$ and $Y$ have the same invariant measure and quadratic variation (informally speaking, the same simulation complexity), but the presence of the $m$-divergence free drift $r\neq 0$ speeds up the convergence to 'equilibrium'.

	This idea has been used to speed up sampling of measure in high-dimension using non-reversible MCMC. There are few rigourous results proving such a phenomenon, in particular when considering the convergence of the empirical measure $\tfrac 1T \int_0^T\delta_{X_s}ds$ to the invariant measure. One way to establish this speed-up effect, consists in proving that the large deviations rate of the empirical measure of $X$ is larger than the one of $Y$, meaning that it is less likely for the empirical measure of $X$ to fluctuate away from the invariant measure. This is rigorously established for diffusions in \cite{rey2015irreversible}.

	To our surprise, Theorem~\ref{t:asympnr} provides a counterexample to this floklore expectation. Empirical measures may be recovered calculating empirical currents on exact forms, a well-known feature of geometric rough paths. Yet, as soon as one enriches the considered set of observables to include the action of currents on \emph{closed} forms, the picture is actually reversed, at least for manifold with some form of flatness (in the sense of Theorem~\ref{t:asympnr}-(2)). Indeed, reversible or more in general homologically reversible processes may achieve equality in \eqref{e:gaussineq}, which is instead a strict inequality for non-homologically reversible processes, see Theorem~\ref{t:asympnr}, (1) $\Rightarrow$ (2).

	To make a trivial example\footnote{In this explicit example, it is easily seen that one may reduce to the reversible case $\bar h=0$ via a rotation that does not effect the convergence speed. That is why we can compare the convergence speed as in \cite{rey2015irreversible} even if the process $X$ is homologically reversible but not reversible for $\bar h\neq 0$}, fix $\bar h\in \bb R^d$ and consider the two processes on a standard flat torus
	\bel{e:yyy}{
		& \dot X = \bar h + \dot W
		\\
		& \dot Y= r(Y) + \dot W
	}
	where $r$ is a divergence-free vector field with $\int r(y)\,dy=\bar h$. Theorem~\ref{t:asympnr} guarantees that the large deviation rate (in the long time limit) of the random homology associated to $X$ is strictly larger than the homology associated to $Y$, while the situation is exactly reversed for the deviations of the empirical measure. This shows that in general non-reversible perturbations may fail to provide a speed-up when sampling invariant observables that are not functions of the empirical measure; otherwise stated $Y$ would in this case be a better choice to sample the volume measure, while $X$ would be a better choice to sample $\bar h$. The authors found this phenomenon non-trivial, unexpected and of some interest for the MCMC community.
\end{remark}

\subsection{Open problems}
\label{ss:open}
In this section we present two open questions. The second one in particular was one of our initial motivations to investigate the problem.

To state the first open question, we start by a corollary that is an easy consequence of \cite[page~57]{Mcmullen2019} and \eqref{e:eps}.
\begin{corollary}
	\label{c:torelli}
	If $d=2$ and $b=0$, there is a one-to-one mapping between Riemannian tensors $g$  on $M$, up to a normalized conformal transformation, and the rate function $G$. Namely, observing the random homology, one can reconstruct the generator of the process up to a random time change of average $1$.
\end{corollary}
In the general reversible case, one may wonder what information one gets on the metric-measure space $(M,g,m)$ by the knowledge of $G(\cdot)$. To answer this question, one may need some weighted version of Torelli's theorem in the form presented by McMullen.

A second question concerns the maximization of $G(\cdot)$ within classes of Riemannian metrics. First one should rule out trivial transformations, as for instance multiplying the metric by a constant just scales $G$ quadratically. However, once trivial transformations are factored out, one may wonder whether manifolds obtained as some rigid extension of manifolds with constant curvature maximize $G$. That would mean, even in the case $b=0$, that breaking the symmetry of loops of a Brownian motion is the least likely on manifolds with constant curvature. Theorem~\ref{t:asympnr} is a result in this direction, although limited to the case of zero curvature. Indeed such a theorem states that  $G(h)/Q(h)\le 1$, with equality holding on flat tori, or manifolds were fluctuations of the homology only arises from a random dynamic on flat tori. Normalizing by $Q(h)$ can be interpreted here as factoring out trivial transformation of the metric. The question however remains open: how to state and prove some maximal properties of manifold with constant curvature, even in the case of (possibly punctured) hyperbolic surfaces?

\section{Technical preliminaries}
\label{s:tools}
\subsection{Lifted generator}
\label{ss:lift}
$d^\ast \colon \mc D^{1}\to \mc D^{0}$ denotes the usual deRham codifferential, and recalling that $m$ is the invariant measure and $V$ was defined in \eqref{e:vr}, we let $d^\ast_m\colon D^{1}\to D^{0}$ be the $m$-weighted codifferential which we can define by duality:
\bel{e:codfm}{
	m (f\,d^\ast_m \omega) \coloneq -m( \langle \nabla f, \omega \rangle) ,\qquad \text{for all $f \in \mc D^0$}
}
or equivalently, as it can be verified via an integration by parts
\bel{e:codfm2}{
	d^\ast_m \omega = d^\ast \omega - \langle \nabla V, \omega \rangle
}
Recall the definition of the generator $L$ in \eqref{e:generator}. Define the lifted generator $\mb L\colon \mc D^1 \to \mc D^0$ as
\bel{e:Lforms}{
	& \mb L \omega : = \frac{1}{2} d^{\ast} \omega + \langle b ,\omega \rangle
	= \tfrac{1}2 d^\ast_m \omega + \langle r ,\omega \rangle
}
With this notation the condition that the measure $e^{-V} \vol$ is invariant, reads after a straightforward computation as
\bel{e:rr}{
	d^\ast_m (gr)=0
}
Moreover, whenever $I(\mu,j)<\infty$ and thus in particular $\mu=\varrho\,m$ with $\sqrt{\varrho} \in W^{1,2}(m)$, the typical current $j_\mu$ as defined at the end of Section~\ref{ss:notation} is written as
\bel{e:jmu2}{
	j_\mu(\omega)= \mu (\mb L\omega)= m \big(- \langle \nabla \sqrt{\varrho}, \sqrt{\varrho} \,\omega \rangle\big)+\mu( \langle r,\omega \rangle)
}
where the integral in $m$ in the right hand side makes sense as both sides of $\langle \cdot , \cdot \rangle $ are in $L^2(m)$.

As $\|\omega\|_{\mu}^2\coloneq \mu(|\omega|^2)$ denoted the $L^2(\mu)$ norm (depending on $g$) for $1$-forms, the dual norm on currents, still denoted by $\|\cdot \|_{\mu}$ as remarked in Section~\ref{ss:notation}, is given by
\bel{e:dualnorm}{
\|j\|_{\mu}^2=\sup_{\omega} \, 2 j(\omega)- \|\omega\|_{\mu}^2
}
For a measure $\mu \in \mc{P}(M)$ and a $\mu$-integrable vector field $E$, we usually denote $\mu E$ the current defined by  $(\mu E)(\omega)=\mu ( \langle E, \omega \rangle)$, in other words $\mu E$ is the current with Radon-Nikodym derivative w.r.t.\ $\mu$ given by $E$ in the sense of \cite[Chapter~3]{diestel1977vector}. With this notation, the scalar product $\sjl \cdot,\cdot \sjd_{\mu}$ inducing the norm $\|\cdot\|_{\mu}$ on currents satisfies
\bel{e:jscalar}{
\sjl j, \mu\,g^{-1}\omega\sjd_{\mu} =  j(\omega)
}
In particular if $I(\mu,j)<\infty$ then, see \eqref{e:jmu2}
\bel{e:jmuscalar}{
	\sjl j,  j_\mu \sjd_{\mu} =  j(-\tfrac 12 d\log \varrho + gr )=j(gr )
}

\subsection{Weighted harmonic forms}
\label{ss:dual}
In this section we discuss some straightforward tilted version of the Hodge decomposition.
\begin{definition}
	\label{d:mharnom}
	A $1$-form $\omega \in \mc D^1$ is \emph{$m$-harmonic} if $d\omega=0$ and $d^\ast_m \omega=0$. It is \emph{$\mb L$-harmonic} if $d\omega=0$ and $d \mb L\omega=0$, see \eqref{e:Lforms}.
\end{definition}
It is easily checked that $m$-harmonic means that $\omega$ is in the kernel of a $m$-weighted Hodge laplacian, from which the name. On the other hand, notice that whenever $r=0$, namely whenever the invariant measure $m$ is reversible, the notion of $\mb L$-harmonicity reduces to the one of $m$-harmonicity. Indeed by definition, if $\omega$ is $\mb L$-harmonic, $\mb L\omega$ is constant, and thus it equals $m(\mb L\omega)$. However if $r=0$ then $\mb L \omega = d^\ast_m \omega$ and $m(d^\ast_m \omega)=0$ for $\omega \in \mc D^1$ and thus $d^\ast_m \omega=0$.

As a weighted version of the Hodge decomposition, it is easy to check that the space of $1$-forms can be decomposed as a direct sum
\bel{e:hodge}{
	\mc D^1= \mc D^{1,\mathrm{exact}} \oplus
	\mc D^{1,m\mathrm{-harm}} \oplus
	\big(\mc D^{1,\mathrm{closed}}\big)^\perp
}
the three components of the direct sum represent respectively exact forms, $m$-harmonic forms and forms that are orthogonal in $L^2(m)$ to closed forms. The three components are orthogonal in $L^2(m)$.

Similarly, we can decompose
\bel{e:hodge2}{
	\mc D^1= \mc D^{1,\mathrm{exact}} \oplus
	\mc D^{1,\mb L\mathrm{-harm}} \oplus
	\big(\mc D^{1,\mathrm{closed}}\big)^\perp
}
where now however the decomposition is not, in general, orthogonal in $L^2(m)$. The relation between the weighted Hodge decomposition \eqref{e:hodge} and \eqref{e:hodge2} is uniquely given as follows. If $\omega=df+\eta$ is closed with $\eta$ $m$-harmonic, then by Fredholm alternative we can solve in the unknown $u \in \mc D^0$ ($u$ is smooth, a standard fact by elliptic regularity, see \cite{hormander1963linear})
\bel{e:xxi}{
	-Lu = \langle r,\eta \rangle - m(\langle r,\eta \rangle)
}
which uniquely determines an exact $1$-form $du$, and $\omega=(df-du)+(\eta+du)$ gives the decomposition \eqref{e:hodge2} for closed forms since $\mb L(\eta+du)=\langle r,\eta \rangle +Lu= m(\langle r,\eta \rangle)$ is constant.

\begin{remark}
	\label{r:isomorphism}
	There are two linear isomorphisms
	\bel{e:isomorphism}{
		H^1(M;\bb R) \ni c\mapsto \eta_c \in \mc D^{1,m\mathrm{-harm}}
		\\
		H_1(M;\bb R) \ni h\mapsto \eta^h \in \mc D^{1,m\mathrm{-harm}}
	}
	associating to a cohomology class $c$ a unique $m$-harmonic form $\eta_c$ in class $c$, and to each homology class $h$ a unique $m$-harmonic form $\eta^h$ such that $m(\langle g^{-1} \eta^h,\eta_c \rangle)=\langle h,c\rangle$ for all $c\in H^1(M;\bb R)$. Such isomorphisms are isometries when $\mc D^{1,m\mathrm{-harm}}$ is regarded as a finite-dimensional subspace of $L^2(m)$ and $H^1(M;\bb R)$, $H^1(M;\bb R)$ are equipped with the scalar products defined at the beginning of section~\ref{ss:gc}.

	Similarly there are two linear isomorphisms
	\bel{e:isomorphism2}{
		H^1(M;\bb R) \ni c\mapsto \omega_c \in \mc D^{1,\mb L\mathrm{-harm}}
		\\
		H_1(M;\bb R) \ni h\mapsto \omega^h \in \mc D^{1,\mb L\mathrm{-harm}}
	}
	associating to a cohomology class $c$ a unique $\mb L$-harmonic form $\omega_c$ in class $c$, and to each homology class $h$ a unique $\mb L$-harmonic form $\eta^h$ such that $m(\langle g^{-1}\omega^h,\omega_c \rangle)=\langle h,c\rangle$ for all $c\in H^1(M;\bb R)$. Such isomorphisms are isometries when $\mc D^{1,m\mathrm{-harm}}$ is regarded as a finite-dimensional subspace of $L^2(m)$ and $H^1(M;\bb R)$, $H^1(M;\bb R)$ are equipped with the scalar products defined at the beginning of section~\ref{ss:gc}.
\end{remark}

\subsection{Weighted minimality and harmonicity of submersions}
\label{ss:fib}

In this section we introduce some basic notions of harmonicity and minimality in the context of weighted Riemannian manifolds. Minimality in this sense has been extensively studied in the last decades, see e.g.\ \cite{cheng2020minimal}, however the authors are not aware of any established connection with harmonicity w.r.t.\ Witten Laplacians of Riemannian submersions. Such a connection is well-established in the case without weight, we refer to \cite{eells1964harmonic} for a classical introductory text. See also
\cite{eells1995two} and
\cites{IliasShouman2018,cheng2015stability}, 
for weighted harmonic maps, and \cites{lichnerowicz1969applications,toth1984toroidal,nagano1975minimal} for Albanese maps.

In this section, $(M, g, m)$ is a smooth compact weighted Riemannian manifold, where the weight $m$ is a measure on $M$ with strictly positive density, $m = e^{-V} \vol$. Of course, for us $m$ is to be thought as the invariant probability associated to \eqref{e:generator}.  $(N, g')$ on the other hand is just a smooth Riemannian manifold with no weight associated.

\subsubsection{Weighted tension field}
\label{sss:tension}
We say that a smooth map $\varphi: M \to N$ is \textit{$m$-harmonic} if it is a critical point of the energy functional (which depends on $m$, $g$ and $g'$)
\begin{equation}
	\label{eq:f-energy_functional}
	E_m(\varphi) = \int |d \varphi|^2 \,dm
\end{equation}
This means that if $\Phi \colon [0,1)\times M\to N$ is smooth with $\Phi(0,\cdot)=\varphi(\cdot)$, then
\bel{e:eels}{
\tfrac{d}{dt} E_m(\Phi(t,\cdot))|_{t=0}=0
}
When $V=0$, this is sometimes called \emph{harmonic in the Eells-Sampson sense}. Notice also that if $d\coloneq \operatorname{dim}(M) \ge 3$, then the notion of $m$-harmonicity can be regarded as a standard notion of harmonicity for the conformal equivalent metric $\hat g= e^{-\frac{2}{d-2} V} g$, see e.g.\ \cite[page~189]{Lichnerowicz1969}. However, we do not pursue this point of view here for reasons that will become apparent later.

As a straightforward generalization of the definition of the tension field in the case without weight, define the \emph{$m$-tension field}
\bel{e:mtension}{
	\tau_m(\varphi) \coloneq \tau(\varphi) - d \varphi(\nabla V)
}
where $\tau(\varphi)$ is the usual tension field of $\varphi$, see \cite[Chapter~I.2]{eells1964harmonic}.
As in the standard case $V=0$, it is not hard to check that $\varphi$ is $m$-harmonic if and only if $\tau_m(\varphi) = 0$, see \cite[page~116]{eells1964harmonic}.


\subsubsection{Weighted minimality of submersions}
In this section, we quickly establish the equivalence of $m$-minimality and $m$-harmonicity for Riemannian submersions, a well-known fact when $m=\vol$.
\begin{definition}
	\label{d:minimality}
	Let $\Sigma$ be a smooth compact manifold and $\imath: \Sigma \to M$ a smooth immersion. Let $H$ be the mean curvature normal field of the immersion, and let $(\nabla V)^\perp(x)$ be the orthogonal projection of $\nabla V(x) \in T_xM$ to the normal bundle of $\imath(\Sigma)$ in $x$.

	The \emph{weighted mean curvature} is the normal field
	\bel{e:wmean}{
		H_m= H+(\nabla V)^\perp
	}
	The immersion $\imath$ is \emph{$m$-minimal} if $H_m$ vanishes identically. If $\Sigma \subset M$, $\Sigma$ is called $m$-minimal whenever the inclusion map is $m$-minimal.
\end{definition}
The previous definition is motivated by the fact that $\imath$ is $m$-minimal if and only if it is a critical point of the volume of the immersion, see \cite[Proposition~2]{cheng2015stability}. In particular, this is the usual notion of minimality if $m$ is the volume measure.

Recall that a surjective submersion $\varphi \colon M \to N$ is a \emph{Riemannian submersion} if for all $x \in M$, the restriction of its differential $(d \varphi)_x \colon (\operatorname{ker} d\varphi_x)^{\perp} \to TN_{\varphi(x)}$ is an isometry. The following lemma generalizes propositon in \cite[Chapter~4-C, page~127]{eells1964harmonic} to the weighted case, and we indeed rely on the proof of the standard case. However, one may also check the statement in coordinates fixing an orthonormal frame on the tangent and normal spaces to the fibers.
\begin{proposition}
	\label{p:m_harmonicity_equiv_m_minimality}
	Let $\varphi \colon M \to N$ be a Riemannian submersion. Then $\varphi$ is $m$-harmonic if and only if all fibers of the submersion are $m$-minimal submanifolds in $M$.
\end{proposition}
\begin{proof}
	Given a point $y \in N$, let $\Sigma_y$ be the fiber $\varphi^{-1}(y)$, and let
	$\imath_y \colon \Sigma_y \to M$ be the inclusion map.
	As well known and easy to check, there is a trivial relation between the tension of the submersion $\varphi$ and the tension of the inclusions of fibers, \cite[Chapter~4-C]{eells1964harmonic}. That is, for all $x\in M$
	\begin{equation}
		\label{eq:eesa_tau_identity}
		\tau(\varphi)(x) = - d \varphi_x (\tau(i_{\varphi(x)})(x))
	\end{equation}
	Now notice that (since $\varphi$ is a submersion) the kernel of $d\varphi_x$ consists of vectors tangent to $\Sigma$ at $x$, so that $d\varphi(\nabla V)= d\varphi( (\nabla V)^\perp)$.	Therefore by \eqref{e:mtension}
	\bel{eq:tau_m_identity}{
		\tau_m(\varphi)(x) &  = \tau(\varphi)(x) - d\varphi(\nabla V)(x) = - d \varphi_x (\tau(i_{\varphi(x)})(x))-
		d\varphi_x( (\nabla V)^\perp(x))
		\\
		& = - d\varphi_x\big(\tau(i_{\varphi(x)})(x)+ (\nabla V)^\perp(x)\big)
	}
	Now, $\tau(i_{\varphi(x)} )=H(x)$ from \cite[Chapter~2-D, page 119]{eells1964harmonic}, and therefore $\tau_m(\varphi)=  - d\varphi(H+ (\nabla V)^\perp)$. However, since $H+ (\nabla V)^\perp=H_m$ is normal, namely it is orthogonal to the kernel of $d\varphi_x$ at any $x$, we have $\tau_m(\varphi)=0$ iff $H_m=0$.
\end{proof}



\subsection{Weighted Albanese map}
\label{ss:albanese}
In this section we introduce a weighted Albanese map $a_m$ defined on a weighted Riemannian manifold $(M,g,m)$, consistently with the above notation. We quickly follow the standard construction of Albanese maps, carefully considering the dependance on the weight $m\in \mc P(M)$.

Let $\tilde{M}$ be universal cover of $M$ with canonical projection $p \colon \tilde{M} \to M$, $p^{\ast}$ denoting the pullback of $1$-forms via $p$. Fix $\tilde x_0\in \tilde M$, set $x_0=p(\tilde x_0)$. Denote the action of the fundamental group on $\tilde M$ by $(s,\tilde y) \ni \pi_1(M,x_0) \times \tilde M \mapsto s\tilde y \in \tilde M$.

For $\omega \in \mc D^{1,\mathrm{closed}}$, $p^\ast \omega$ is exact on $\tilde{M}$, namely there exists a unique smooth $u_\omega \colon \tilde M \to \bb R$ such that $du_\omega = p^\ast \omega$ and $u_\omega(\tilde x_0)=0$. For $s\in \pi_1(M,x_0)$, it holds $u_\omega(s\tilde y)-u_\omega(\tilde y)$ is independent of $\tilde x_0$ and $\tilde y$, since its differential vanishes. Moreover, $u_{df}(s\tilde y)-u_{df}(\tilde y)=0$, and $\omega \mapsto u_\omega$ is linear. This entails that there exists an abelian representation $\pi_1(M)\ni s\mapsto h_s \in H_1(M;\bb R)$ such that
\bel{e:hs}{
u_{df+\eta_c}(s \tilde y)-u_{df+\eta_c}(\tilde y)= \langle h_s,c\rangle
}
for $\tilde y\in \tilde M$, $f\in D^0$, $c\in H^1(M;\bb R)$ and $\eta_c$ the unique $m$-harmonic form in class c, see Remark~\ref{r:isomorphism}.

Recalling Remark~\ref{r:isomorphism} one can define
\bel{eq:tilde_J}{
	& \tilde{a}_m \colon \tilde{M} \to (\mc D^{1,m\mathrm{-harm}})^\dagger \simeq H_1(M;\bb R)
	\\
	& \tilde{a}_m(\tilde{y})(\eta_c) \coloneq u_{\eta_c}(\tilde{y}) - u_{\eta_c}(\tilde{x}_0) 
}
It follows from \eqref{e:hs} that for if $\tilde y,\tilde y'\in \tilde M$ cover the same point, namely $\tilde y'=s\tilde y$, it holds $(\tilde{J}(\tilde{y}')- \tilde{J}(\tilde{y}))(\eta_c) = \langle h_s, c\rangle$. Now $G=(h_s)_{s\in \pi_1(M)}$ is a discrete lattice of full rank of $H_1(M;\bb R)$, so that $H_1(M;\bb R)/G \simeq \bb T^{b_1}$ is a flat torus, where $b_1$ is the first Betti number of $M$.

Therefore there is a map $a_m$ such that the following diagram commutes
\begin{equation*}
	\begin{tikzcd}
		\tilde{M} \arrow{r}{\tilde{a}_m} \arrow[swap]{d}{p} & {H_1(M;\bb R)} \arrow{d}{/G} \\%
		M \arrow{r}{a_m}& \bb T^{b_1}
	\end{tikzcd}
\end{equation*}

\begin{remark}
	\label{r:Albanese_map_is_harmonic}
	The weighted Albanese map $a_m \colon M \to \bb T^{b_1}$ is $m$-harmonic (in the weighted Eells-Sampson sense).
\end{remark}
\begin{proof}
	In this proof we denote by $\tilde m$ the lift of $m$ to $\tilde{M}$, and $d,d^\ast_{\tilde m}$ the corresponding operators on $M$ built as in \eqref{e:codfm}. Whenever the arrival space $N$ of a map $\varphi$ as in Section~\ref{sss:tension} is a flat torus, the $m$-harmonicity equation $\tau_m(\varphi)=0$ is actually linear (since Christoffel symbols vanish, as one may easily check with similar computations as in \cite[page~109]{eells1964harmonic}) and can be checked componentwise. So it is enough to show that $\tilde a_m$ is $\tilde{m}$-harmonic as a $\bb R^{b_1}$-valued function.

	For $u_\omega$ as above, notice that $d^\ast_{\tilde m} d u_\omega(\tilde x)=(d^\ast_m \omega)(p(\tilde x))$. In particular $\omega$ is $m$-harmonic iff $u_\omega$ is $\tilde m$-harmonic. So that the $\tilde{m}$-harmonicity of $\tilde a_m$ follows straightforwardly from its definition \eqref{eq:tilde_J}.
\end{proof}
We refer to $a_m$ as the $m$-weighted Albanese map.

\section{Quadratic bounds}
\label{s:proofquad}
In this section we prove Proposition~\ref{p:bound}, although we start with showing a simple statement claimed in Section~\ref{s:main}.
\begin{proof}[Proof of Remark~\ref{r:ldh}]
	Fix an isomorphism $c\mapsto \xi_c$ as above and consider the map $(\mu,j)\mapsto h^{j}\in  H_1(M;\bb R)$ defined by duality as
	\bel{e:hmuj}{
		\langle h^{j},c \rangle= j(\xi_c), \qquad \forall c\in H^1(M;\bb R)
	}
	Since $(\xi_c)_{c\in H^1(M;\bb R)}$ is finite-dimensional, this map is continuous and moreover $h_T\equiv h^{J_T}$, see \eqref{e:randomh}. By contraction principle \cite[Chapter~4.2.1]{dembo2010} we get that the law of $h_T$ satisfies a large deviations principle with speed $T$ and rate given by
	\bel{e:gpre}{
		H^1(M;\bb R) \ni h \mapsto \inf \big\{ I(\mu,j),\,(\mu,j) \st h^{j}=h \big\}
	}
	However, since we can restrict to $(\mu,j)$ with $I(\mu,j)<\infty$, it follows that $j$ is closed, and thus the relation $h^{j}=h$ in \eqref{e:gpre} is equivalent to $j(\omega)= \langle h,c\rangle$ for all $\omega \in [c]$, regardless of the original linear isomorphism $c\mapsto \xi_c$ used to define $h_T$.
\end{proof}

Recall that $I$ and $G$ were introduced in Definition~\ref{d:h}.
Recall that $\HH$ was introduced in Definition~\ref{d:h}, and for $h\in H_1(M;\bb R)$, denote
\bel{e:hh}{
	\HH_h\coloneq \big\{ (\mu,j)\in \HH \st j(\omega)= \langle h,c\rangle,\,\text{for all $\omega \in [c]$ and $c\in H^1(M;\bb R)$}\big\}
}
So that $G(h)=\inf_{(\mu,j)\in \HH_h} I(\mu,j)$.

The following remarks are immediate.
\begin{remark}
	\label{r:good}
	Recall that $\mc P(M)$, respectively $\mc J(M)$, are equipped with the weakest topology such that $\mu \mapsto \mu(f)$ is continuous for all $f\in \mc D^0$, respectively $\omega \mapsto j(\omega)$ is continuous for all $\omega \in \mc D^1$. Then the set $\{(\mu,j)\in \mc P(M)\times \mc J(M)\st I(\mu,j)\le k \}$ is compact for each $k\ge 0$. In other words, $I$ is \emph{good} in the sense \cite[Chapter~1.2]{dembo2010}.
\end{remark}

\begin{remark}
	\label{r:hbar}
	Since $m(\langle r,df\rangle)=0$ for $f\in \mc D^1$, it remains defined the rotation number of the current $mr$, namely an element $\bar h \in H_1(M;\bb R)$ such that
	\bel{e:hbar4}{
		m(\langle r,\omega \rangle) = \langle \bar h,c\rangle \qquad \text{for all $\omega \in [c]$ and $c\in H^1(M;\bb R)$}
	}
\end{remark}

\begin{lemma}
	\label{l:upboundgauss}
	For $h\in H_1(M;\bb R)$, let $\mc J_h(M)$ be the space of closed currents with rotation number $h$, that is
	\bel{e:jh}{
		\mc J_h(M)\coloneq \left\{j\in \mc J(M)\st j(\omega)= \langle h,c \rangle, \,\text{for all $\omega \in [c]$ and $c\in H^1(M;\bb R)$}  \right\}
	}
	For $Q(\cdot)$ as defined in \eqref{e:gaussineq}, it holds
	\bel{e:qeq}{
		Q(h)\coloneq \inf_{j \in \mc J_h(M)} I(m,j)
	}
\end{lemma}
\begin{proof}
	For $j \in \mc J_h(M)$, $(m,j)\in \HH$ and thus by \eqref{e:rate} and \eqref{e:jmu2} with $\varrho \equiv 1$
	\bel{e:mj1}{
		I(m,j)=
		& \sup_{\omega \in \mc D^1} (j-j_m)(\omega)-\tfrac 12 m(|\omega|^2)
		\\
		= & \sup_{\omega \in \mc D^1} j(\omega)-m(\langle r,\omega\rangle)-\tfrac{1}{2}m(|\omega|^2)
		\\
		= & \sup_{\omega \in \mc D^{1,m\mathrm{-harm}}
			\oplus (\mc D^{1,\mathrm{closed}})^\perp}
		j(\omega)-m(\langle r,\omega\rangle)-\tfrac{1}{2}m(|\omega|^2)
	}
	where the second equality follows from $m(d_m^\ast \omega)=0$, while the last equality follows from the orthogonality of the direct sum \eqref{e:hodge} in $L^2(m)$, and $j(df)=m(\langle r, df\rangle)=0$, so that the supremum is attained on $1$-forms $\omega$ with vanishing exact term in the decomposition. Recalling that $\eta_c$ is the unique $m$-harmonic form in cohomology class $c$ and that  $(\eta_c)_{c\in H^1(M;\bb R)} =\mc D^{1,m\mathrm{-harm}}$ we obtain
	\bel{e:mj2}{
		I(m,j) = &  \sup_{c \in H^{1}(M;\bb R)}
		j(\eta_c)-m(\langle r,\eta_c \rangle)-\tfrac{1}{2}m(|\eta_c|^2)
		\\  &
		+ \sup_{\xi \in  (\mc D^{1,\mathrm{closed}})^\perp} j(\xi)-m(\langle r,\xi \rangle)-\tfrac{1}{2}m(|\xi|^2)
	}
	Now, for $\bar h$ as in Remark~\ref{r:hbar}, and since $d^\ast_m \eta_c=0$ we actually have
	\bel{e:homologytypical}{
		\langle \bar h,c \rangle = 	m(\langle r,\eta_c \rangle)
		=m\big(\langle -\tfrac12 \nabla \varrho - \tfrac 12 \nabla V +r,\eta_c \rangle\big)= j_m(\eta_c)	}
	so $\bar h$ is nothing but the rotation number of the typical current in the invariant measure $j_m$.

	Moreover $m(|\eta_c|^2)=(c,c)$, see Remark~\ref{r:isomorphism}, $j(\eta_c)=\langle h,c\rangle$ for all $j\in \mc J_h(M)$. Thus we deduce from \eqref{e:mj2}
	\bel{e:qeq2}{
		\inf_{j \in \mc J_h(M)} I(m,j)  = &
		\sup_{c \in H^{1}(M;\bb R)}
		\langle h - \bar h,c \rangle-\tfrac{1}{2}(c,c)
		\\ &
		+\inf_{j \in \mc J_h(M)} \sup_{\xi \in  (\mc D^{1,\mathrm{closed}})^\perp} j(\xi)
		-m(\langle r,\xi \rangle)-\tfrac{1}{2}m(|\xi|^2)
	}
	Now, since the scalar products induced on $H^1(M;\bb R)$ and $H_1(M;\bb R)$ are in duality, the r.h.s.\ in the first line of \eqref{e:qeq2} is exactly $\tfrac 12 (h-\bar h,h-\bar h)$. On the other hand, the second line is nonnegative (as one can always take $\xi=0$ in the sup) and equals $0$ for any current whose restricition to $ (\mc D^{1,\mathrm{closed}})^\perp$ coincides with $m\,r$. \eqref{e:qeq} is thus proved, the infimum being achieved at the current $j$ given by $j(df+\eta_c+\xi)= \langle h,c\rangle + m(\langle r,\xi \rangle)$, where $df+\eta_c+\xi$ represents the decomposition of a generic $1$-form in $\mc D^1$ as in \eqref{e:hodge}.
\end{proof}

\begin{proof}[Proof of Proposition~\ref{p:bound}, formula \eqref{e:gaussineq}]
	The inequality \eqref{e:gaussineq} is an immediate consequence of Lemma~\ref{l:upboundgauss} since
	\bel{e:upboundgauss2}{
		G(h)= \inf_{(\mu,j) \in \HH_h} I(\mu,j) \le \inf_{j \in \mc J_h(M)} I(m,j)= Q(h)
	}
\end{proof}

\begin{proof}[Proof of Proposition~\ref{p:bound}, formula \eqref{e:eps}]
	We start with the proof of the $\le$ inequality in \eqref{e:eps}. Fix $h\in H_1(M;\bb R)$ and for $\omega^h$ as in Remark~\ref{r:isomorphism}, let $u \equiv u^h \in \mc D^0$ be the unique solution to
	\bel{e:fred2}{
		Lu - 2 \langle r,du \rangle = d^\ast_m \omega^h
	}
	with $m(u)=0$. The operator $L^\dagger u \coloneq Lu - 2 \langle r,du \rangle$ is the adjoint of $L$ in $L^2(m)$, so that the well-posedness of the equation \eqref{e:fred2} is a consequence of the Freedholm alternative and the fact that the r.h.s.\ integrates to $0$ w.r.t.\ $m$.

	For $\eps>0$ such that $\eps \|u\|_{C(M)} < 1$, $1+\eps u$ is a smooth probability density, since $m(u)=1$. Then define
	\bel{e:optimalmuj}{
		& \nu^{\eps,h}=(1+\eps u) m \,\in \mc P(M)
		\\
		& \imath^{\eps,h}(\omega)= j_{\nu^{\eps,h}}(\omega)+\eps\, m \left( \langle g^{-1} \omega^h,\omega \rangle \right)
	}
	By these definitions
	\bel{e:closed}{
		& \imath^{\eps,h}(df) =\nu^{\eps,h}(Lf)-\eps\, m(f d^\ast_m \omega^h)
		\\
		& \phantom{\imath^{\eps,h}(df) } = \eps \left(m( u Lf- d^\ast_m \omega^h) \right)=\eps \,m\left(  f (L^\dagger u - d^\ast_m \omega^h)\right)=0
		\\
		& \imath^{\eps,h}(\omega_c)= \nu^{\eps,h}(\mb L \omega_c)+\eps m(\langle g^{-1} \omega^h,\omega_c \rangle)= \langle \bar h,c\rangle +\eps \langle h,c\rangle
	}
	where in the first identity we used  \eqref{e:codfm}, \eqref{e:jmu2}, the invariance of $m$ and \eqref{e:fred2}; in the second identity we used $(\mb L \omega_c)(x)=m (\langle r,\omega_c\rangle)= \langle \bar h,c\rangle$ for every $x\in M$ and the definition of $\omega^h$ in Remark~\ref{r:isomorphism}.

	\eqref{e:closed} imply that $(\nu^{\eps,h},\imath^{\eps,h}) \in \HH_{\bar h +\eps g}$ and therefore by the very definition of $G$ \eqref{e:rate2} and \eqref{e:rate}
	\bel{e:upboundgauss3}{
		G(\bar h +\eps h) & \le I(\nu^{\eps,h},\imath^{\eps,h})= \frac{\eps^2}2 \int \!\! \frac{|\omega^h|^2}{(1+\eps u)^2}  d\nu^{\eps,h}
		= \frac{\eps^2}2 \int \!\! \frac{|\omega^h|^2}{(1+\eps u)}  dm
	}
	so that, taking the limit inside the integral by bounded convergence
	\bel{e:upboundgauss4}{
		\varlimsup_{\eps} \eps^{-2} G(\bar h +\eps h) \le \tfrac 12 m(|\omega^h|^2)=(h,h)_r
	}

	We next turn to the $\ge$ inequality in \eqref{e:eps}. 	Since $I(\cdot,\cdot)$ has compact sublevel sets, see Remark~\ref{r:good}, for each $h\in H_{1}(M;\bb R)$ and $\eps>0$, there exists a $(\mu^{\eps,h},j^{\eps,h}) \in \HH_{\bar h+\eps h}$ such that
	\bel{e:imuj2}{
	I(\mu^{\eps,h},j^{\eps,h}) = G(\bar h+\eps h) \le \eps^2 \eps^{-2} Q(\bar h+\eps h) = \frac{\eps^2}2 (h,h)
	}
	where in the last inequality we used \eqref{e:gaussineq}.
	In particular $(\mu^{\eps,h},j^{\eps,h})_{0<\eps<1}$ lies in a compact set and any limit point $(\mu,j)$ as $\eps \downarrow 0$ satisfies $I(\mu,j)\le \varliminf_{\eps} I(\mu^{\eps,h},j^{\eps,h})=0$ by \eqref{e:imuj2}. Since $(m,j_m)$ is the unique zero of $I$, it follows $\mu^{\eps,h} \to m$ as $\eps \downarrow 0$. Therefore, recalling that $\eta^h$ is defined in Remark~\ref{r:isomorphism}, and \eqref{e:dualnorm}, \eqref{e:rate}, we get for \emph{any} $c\in H^1(M;\bb R)$ and $f\in \mc D^0$
	\bel{e:gacont}{
	G(\bar h+\eps h) =
	&   I(\mu^{\eps,h},j^{\eps,h})
	=
	\sup_{\omega \in \mc D^1} (j-j_{\mu^{\eps,h}})(\omega)- \tfrac 12 \mu^{\eps,h}(|\omega|^2)
	\\
	\ge &
	(j-j_{\mu^{\eps,h}})(\eps(\eta_c+df))- \tfrac 12 \mu^{\eps,h}(|\eps(\eta_c+df)|^2)
	\\
	= & \langle \bar h +\eps h, \eps c \rangle - \eps j_{\mu^{\eps,h}}(\eta_c+df)
	- \tfrac{\eps^2}2 \mu^{\eps,h}(| \eta_{c}+df|^2)
	}
	where in the inequality we just chose $\omega=\eps (\eta_c+df)$, and we used $j\in \mc J_{\bar h +\eps h}$ in the last line.

	Now notice that $ j_{\mu^{\eps,h}}(\eta_c+df)= \mu^{\eps,h}(\langle r,\eta_c \rangle + Lf) $, $\langle \bar h, c \rangle = j_m(\eta_c)$ and $m(Lf)=m(\langle df, \eta_c\rangle )=0$ to get from \eqref{e:gacont}
	\bel{e:gacont2}{
		\eps^{-2} G(\bar h+\eps h) \ge &
		\langle h, c \rangle
		- \eps^{-1} \mu^{\eps,h}\big(\langle r,\eta_c \rangle +Lf\big)
		\\ &
		+\eps^{-1} m\big(\langle r,\eta_c \rangle +Lf\big)
		- \tfrac 12 \mu^{\eps,h} \big(|\eta_c+df|^2   \big)}
	We then choose $f\in \mc D^0$ as the unique solution to
	\bel{e:fred}{
		-Lf = \langle r,\eta_c \rangle - m( \langle r,\eta_c \rangle)
	}
	with $m(f)=0$, namely \eqref{e:xxi} with $\eta_c$ in place of $\eta$. With such a choice eof $f$ the terms in $\eps^{-1}$ in \eqref{e:gacont2} vanish in view of
	\bel{e:extraterm}{
		\mu^{\eps,h}\big(\big(\langle r,\eta_c \rangle +Lf\big)\big)
		=  m( \langle r,\eta_c \rangle)= m\big(\big(\langle r,\eta_c \rangle +Lf\big)\big)
	}
	Moreover, as Remarked after \eqref{e:xxi}, for $f$ as in \eqref{e:fred}, $\eta_c+df=\omega_c$. Thus by the very definition of $(\cdot,\cdot)_r$, see the discussion at the  beginning of Section~\ref{ss:gc},
	\bel{e:ccr}{
		m(| \eta_{c}+df|^2)=(c,c)_r
	}
	Finally recalling that $\mu^{\eps,h}$ converges weakly $m$ as proved after \eqref{e:imuj2}, since $|\omega_c|^2$ is smooth
	\bel{e:muepshconv}{
		\lim_{\eps} \mu^{\eps,h} \big(|\eta_c+df|^2   \big) = m\big(|\eta_c+df|^2   \big) =(c,c)_r
	}
	Thus passing to the limit in \eqref{e:gacont2}, and using \eqref{e:extraterm}-\eqref{e:muepshconv}
	\bel{e:gacont3}{
		\varliminf_{\eps} \eps^{-2} G(\bar h+\eps h) \ge &
		\langle h, c \rangle  - \tfrac{1}2 (c,c)_r
	}
	As we optimize over $c\in H^1(M;\bb R)$ we get the $\ge$ inequality in \eqref{e:eps}.
\end{proof}

\begin{proof}[Proof of Proposition~\ref{p:bound}, formula \eqref{e:gc}]
	Let $\tilde \HH \subset \HH$ be the set of pairs  $(\mu,j)$ with $\mu=\varrho \,m$, $j=\mu E$ for some smooth, strictly positive density $\varrho$ and smooth tangent vector field $E$. Then, with the notation introduced at the beginning of Section~\ref{s:tools}
	\bel{e:jmuone}{
		j_{\mu}(\omega) =-\tfrac 12 m(\langle \nabla\varrho,\omega \rangle)+ \mu (\langle r,\omega\rangle)
		= -\tfrac 12 \mu(\langle \nabla \log \varrho,\omega \rangle)+ \mu (\langle r,\omega\rangle)
	}
	that is $j_{\mu}=  \mu \,g^{-1}\bar \omega_\varrho$ for $\bar \omega_\varrho = -\tfrac 12 d \log \varrho + gr \in \mc D^1$. Thus (see also \eqref{e:jmuscalar})
	\bel{e:gc1}{
	I(\mu,j)-I(\mu,-j) &
	= \frac 12 \|j-j_\mu\|_{\mu}^2 - \frac 12 \|-j-j_\mu\|_{\mu}^2=-2 \sjl j, j_\mu\sjd_\mu
	\\
	& =- 2 \,j(\bar \omega_\varrho) =  j (d \log \varrho)- 2 j(g\,r)= - 2 j(g\,r)
	}
	where in the last line we used \eqref{e:jscalar} and the fact that $j$ is closed.

	It is easy to check that $\tilde \HH$ is $I$-dense in $\HH$, namely that for $(\mu,j) \in \HH$ there exists a sequence $(\mu_n,j_n) \to (\mu,j)$ with $(\mu_n,j_n)\in \tilde \HH$ and $\lim_n I(\mu_n,\pm j_n)=I(\mu,\pm j)$. Indeed, $I$ is nothing but the lower-semicontinuous envelope of its restriction to $\tilde \HH$. Therefore \eqref{e:gc1} holds on $\HH$ and not just on $\tilde{\HH}$.

	We need to show that, in the quasi-reversible case see Definition~\ref{d:quasireversible}-(b), it holds ($\HH_h$ is defined in \eqref{e:hh})
	\bel{e:gcrestated}{
		\inf_{(\mu,j)\in \HH_h} I(\mu,j)= 	   \inf_{(\mu,j)\in \HH_{-h}} I(\mu,j) - \langle h,\bar c\rangle
	}
	Notice that $(\mu,j) \in \HH_h$ iff $(\mu,-j) \in \HH_{-h}$. Therefore
	\bel{e:gc2}{
		G(-h)
		&
		= \inf_{(\mu,j) \in \HH_{-h}} I(\mu,j)
		= \inf_{(\mu,j) \in \HH_{h}} I(\mu,-j)
		\\
		&  = \inf_{(\mu,j)  \in \HH_{h}} I(\mu,-j)  -2j(gr )+ 2j(gr )
		=\inf_{(\mu,j)  \in \HH_{h}} I(\mu,j)+ 2j(gr )
	}
	where in the last line we used \eqref{e:gc1}. By hypotheses of quasi-reversibility, the vector field $b$ is such that $gb$ is closed, thus $gb \in [\bar c]$ for some cohomology class $\bar c$. However, since $r=b+\tfrac 12 \nabla V$, it holds $gr \in [\bar c]$ as well. Therefore, for each $(\mu,j) \in \HH_h$, the quantity $2 j(gr)$ in the last line of \eqref{e:gc2} equals $2\langle h,\bar c \rangle$ and so it gets out of the inf to get \eqref{e:gc} (which trivially holds for $Q$).
\end{proof}

\section{Asymptotically Gaussian homology}
\label{s:proofag}
In this section we prove Theorem~\ref{t:asympnr}. Recall that $L$ has asymptotically Gaussian homology if $G(h)=Q(h)$, see \eqref{e:gaussineq}.
\begin{lemma}
	\label{l:costantl}
	Assume that $L$ has asymptotically Gaussian homology. Then $L$ is homologically reversible, and moreover $m$-harmonic forms have constant length. That is $|\eta_c|$ is constant (independent of $x$) for all $c\in H^1(M;\bb R)$.
\end{lemma}
\begin{proof}
	Recall the definition \eqref{e:hh} of $\HH_h$. From Lemma~\ref{l:upboundgauss}, $G=Q$ iff for all $h\in H_1(M;\bb R)$, there exists $j^h \in \mc J_h$ such that $I(m,j^h) \le I(\mu,j)$
	for all $(\mu,j) \in \HH_h$. Indeed, the minimizer of the coercive functional $Q(\cdot)=I(m;\cdot)$ over the closed set $\mc J^h$ exists, so that equality holds iff
	\bel{e:gingi}{
		G(h)= \inf_{(\mu,j)\in \HH_h} I(\mu,j)=I(m,j^h)=Q(h)
	}
	In such a case, since $j_m=m\,r$, it holds necessarily that $j^h=m(r+g^{-1}\xi^h)$ for some $\xi^h \in \mc D^1$, $\xi^h \in L^2(m)$. Moreover, $\xi^h$ has to be orthogonal to exact forms in $L^2(m)$ since $j^h(df)=m(\langle r,df \rangle)=0$`'.
	Now take in \eqref{e:gingi} $\mu= m (1+u)$  for $u\in \mc D^0$ with $u\ge -1$ and $\int u \,dm=0$, and $j$ of the form $j^h+m g^{-1}\zeta$ for some $\zeta \in \mc D^1$ with $m g^{-1}\zeta \in \mc J_0$. To get from \eqref{e:rate}-\eqref{e:jmu2}, for all $u$'s and $\zeta$'s as just described
	\bel{e:lagrange}{
		I(m,j^h) \le I(m (1+u),j^h+ m g^{-1} \zeta )
		= \frac 12 \int  \frac{ |\xi^h+\zeta - \tfrac{1}{2} d u - u \,r|^2}{1+u}   \,dm
	}
	Now, changing $u$ to $\eps u$, $\zeta$ to $\eps' \zeta$, since we chose $u$ and $\zeta$ smooth, it is easily seen that the r.h.s.\ of \eqref{e:lagrange} is differentiable in $\eps,\eps'$, and imposing that the derivatives must vanish at $\eps=\eps'=0$ one gathers
	\bel{e:lagrange2}{
		& m(\langle \xi^h,\zeta\rangle)=0
		\\
		& m\left( -|\xi^h|^2 u + \langle \nabla u +2 u r,\xi^h \rangle \right)=0
	}
	The first equation holds for all smooth $\zeta$ with $m g^{-1}\zeta \in \mc J_0$, so that this equation implies that $\xi^h$ is closed. And since $\xi^h$ is orthogonal in $L^2(m)$ to exact forms, it must hold $\xi^h=\eta^h$, for $\eta^h$ as in Remark~\ref{r:isomorphism}. In particular the term  $m\left( \langle \nabla u ,\eta^h \rangle \right)$ vanishes and we get from the second equation, recalling that $\int u\, dm=0$
	\bel{e:lagrange3}{
		|\eta^h|^2+ 2 \langle r,\eta^h \rangle =\mathrm{constant} \qquad \text{for all $h\in H_1(M;\bb R)$}
	}
	By polarization, then one easily gets that the quadratic term and the linear one in $\eta^h$ must be independently costant. As $h$ spans $H^1(M;\bb R)$, $\eta^h$ spans $\mc D^{1,m\mathrm{-harm}}$ so that we get that $\langle r,\eta\rangle$ is constant (thus $L$ is homologically reversible) and $|\eta|$ are constant for all $m$-harmonic $\eta$'s.
\end{proof}

\begin{lemma}
	\label{l:constantminimal}
	Let $(M,g,m)$ be a weighted Riemannian manifold as in Section~\ref{ss:fib}. If every $m$-harmonic form on $M$ has constant length, then the weighted Albanese map $a_m$ defined in Section~\ref{ss:albanese}, is a Riemannian submersion with $m$-minimal fibers, according to Definition~\ref{d:minimality}.
\end{lemma}
\begin{proof}
	$a_m$ is always $m$-harmonic, see Remark~\ref{r:Albanese_map_is_harmonic}. On the other hand, $m$-harmonic submersions have $m$-minimal fibers, see Proposition~\ref{p:m_harmonicity_equiv_m_minimality}. So it is enough to check that if every $m$-harmonic form has constant length, $a_m$ is a Riemmanian submersion.  If every $m$-harmonic form on $M$ has constant length, by polarization, all pointswise scalar products $\omega \cdot \omega' = \langle g^{-1}\omega,\omega'\rangle$ in $T^\ast_x M $ are constant in $x$ for $\omega,\omega'$ $m$-harmonic forms. Fix an orthonormal base of the $b_1$-dimensional space of $m$ harmonic equipped with the Hilbert norm $\|\omega\|_m= \sqrt(m(|\omega|^2))$. Since scalar products are constant, it follows that the base is \emph{pointwise} orthonormal on each $T_x^\ast M$, not just in $L^2(m)$. Since $a_m$ pulls back harmonic forms on $\bb T^1$ to $m$-harmonic forms on $M$, orthonormal coframes are pulled back to orthonormal coframes in $L^2(m)$, and thus pointwise orthonormal coframes. This is equivalent to $a_m$ being a Riemannian submersion.
\end{proof}

\begin{lemma}
	\label{l:2implies3}
	Let $\psi\colon M\to \bb T^d$ be of class $C^2$, and suppose that $\psi$ pushes forward the Riemannian metric on $M$ to a flat metric on the torus $\bb T^d$. Then the semimartingale $Y_t \coloneq \psi(X_t)$ has quadratic variation $[Y,Y]_t= S\,t$ for some constant (symmetric, positive definite) matrix $S$ and $Y_t-\int B(X_s)\,ds$ is a martingale, where $B(x) \in T_{\psi(x)}\bb T^d$ is characterized as follows. For $\mc O\ni x$ a small enough open ball in $M$, one can write on $\mc O$, using an orthonormal (w.r.t.\ the induced flat metric) frame on $\psi(\mc O)$: $B=(B^1,\ldots, B^d)$, $\psi=(\psi^1,\ldots,\psi^d)$. Then $B^k=L\psi^k$ where $L$ is the generator \eqref{e:generator}.
\end{lemma}
\begin{proof}
	Let us compute in coordinates in $\mc O$ using the orthonormal frame as in the statement of the lemma. The statement on the quadratic variation is trivial, since $Y$ satisfies, using standard semimartingale notation, $d[Y^h,Y^k]_t= S^{h,k}(X_t)\,dt$, with $S^{h,k}(x)= g^{i,j}(x) \partial_i \psi^h(x) \partial_j \psi^k(x)$. This is constant in $x$, since it is nothing but the pushforwarded metric on $\bb T^d$, which is flat by hypotheses. On the other hand, the bounded variation term in the Doob decomposition can be carefully computed in coordinates to get that $Y_t-\int_0^t B(X_s)ds$ is a martingale for
	\bel{e:bigb}{
		B^k(x)= L\psi^k(x) - \tfrac 12  \sqrt{|S(x)|}\partial_i \left(  \tfrac{1}{\sqrt{|S(x)|}}  S^{i,k}(x) \right)
	}
	where $|S|$ denotes the determinant. As already noticed, $\psi$ being a submersion implies that $S(\cdot)$ is constant (actually the identity in our coordinates). Therefore the Riemannian correction term in the last formula (the last term involving derivatives of $S$) vanishes.
\end{proof}

We are finally ready to prove the main theorem.
\begin{proof}[Proof of Theorem~\ref{t:asympnr}]

	(1) $\Rightarrow$ (2). From Lemma~\ref{l:constantminimal}, the map $a_m$ that we defined in Section~\ref{ss:albanese}, is a Riemannian submersion with minimal fibers. A result by Hermann, see \cite[Theorem~9.3]{besse2007einstein}, states that Riemannian submersions whose total space is complete are locally trivial fiber bundles. Since $M$ is compact, thus geodesically complete, $(M,\bb T^{b_1},a_m)$ is a locally trivial fiber bundle. Fibers are then $m$-minimal still from Lemma~\ref{l:constantminimal}, while the homological reversibility of $L$ comes from Lemma~\ref{l:costantl}.

	\smallskip

	(2) $\Rightarrow$ (3). By hypotheses there exists a smooth $\phi\colon M\to \bb T^{b_1}$ and a flat metric on $\bb T^{b_1}$ such that $(M,\bb T^{b_1},\phi)$ is a locally trivial fiber bundle and $\phi$ is a submersion with $m$-minimal fibers, see Section~\ref{ss:fib}. In particular, by Proposition~\ref{p:m_harmonicity_equiv_m_minimality}, $\phi$ is $m$-harmonic\footnote{In the statement of Theorem~\ref{t:asympnr} we did not detail the smoothness assumptions on the projection map $\phi$. In the literature, it may be sometimes assumed smooth or just differentiable. However notice that the harmonicity of $\phi$ guarantees that this two conditions are actually equivalent.}, in the weighted Eells-Sampson sense, see Section~\ref{sss:tension}. Since $\bb T^{b_1}$ is flat (in particular Christoffel's symbols vanish), using a orthonormal frame as in Lemma~\ref{l:2implies3}, it is easily seen that the components $\phi^k$ are $m$-harmonic in the sense $\Delta_m \phi^k=0$. In particular from Lemma~\ref{l:2implies3} applied with $\psi=\phi$, we get that \emph{locally} $B^k(x)= L \phi^k= \frac 12 \Delta_m \phi^k + \langle r,d\phi^k\rangle = \langle r,d\phi^k\rangle$, for $k=1,\ldots,b_1$. Now locally $d\phi^k$ is an $m$-harmonic form, being the differential of a $m$-harmonic function. In particular, since $b$ is homologically reversible by hypotheses (2), $\langle r,d\phi^k\rangle\eqcolon \tilde{h}^k$ is constant in $x$ in any small enough open set, and thus everywhere on $M$, since $M$ is connected. In other words, still by Lemma~\ref{l:2implies3}, $Y_t- \tilde{h}t$ is a continuous martingale with quadratic variation coinciding with the quadratic variation of a standard (flat) Brownian motion. Namely the statement.

	\smallskip

	(3) $\Rightarrow$ (1).	First notice that the quadratic variation $\big[\phi(X),\phi(X)\big]_t$ equals $I_{b_1} t$ where $I_{b_1}$ is the identity (in orthonormal coordinates). This implies that $b_1\le \mathrm{dim}(M)$ and that $d\phi$ has maximal rank, that is $b_1$. Thus $\phi$ is a submersion (although not necessarily a Riemannian submersion).

	If $e$ is a smooth $1$-form on $\bb T^{b_1}$, then denoting $\phi_\ast$ the pullback on forms
	\bel{e:changeofvariable}{
		\int_0^t e(Y_s)\circ dY_s = \int_0^t (\phi_\ast e)(X_s)\circ dX_s
	}
	If $e$ is harmonic on the flat torus $\bb T^{b_1}$ and $Y$ satisfies \eqref{e:sdetorus}, it is easy to see that the l.h.s.\ of \eqref{e:changeofvariable} is Gaussian for every $t\ge 0$, since it coincides in law with $c\dot W_t + \langle c,\bar h\rangle$ where $c$ is the cohomology class of $e$ and $W$ a standard Brownian motion on $\bb R^{b_1}$. On the other hand, $\phi_\ast e$ is closed in $M$ since pullbacks commute with differentials. As we have already noticed that $\phi$ is a submersion, any cohomology classes $c\in H^1(M;\bb R)$ have a representative closed $1$-form  $\xi_c$ of the type $\phi^\ast e$ for $e$ harmonic on $\bb T^{b_1}$. Thus the random homology $h_T$ defined by $\langle h_T,c \rangle = J_T(\xi_c)$ is Gaussian with covariance $|c|^2 T$, and thus has Gaussian large deviations with rate $Q(\cdot)$, see \eqref{e:gaussineq}. Since the large deviations rate of $h_T$ does not depends on this choice of the isomorphism $c\mapsto \xi_c$, see Remark~\ref{r:ldh}, we conclude.
\end{proof}

\bibliographystyle{plain}
\bibliography{biblio}

\end{document}